\pdfoutput=1
\documentclass[11pt,leqno]{article}

\usepackage{amsmath,amsthm,amssymb}
\usepackage{graphicx,subfigure,booktabs,enumerate,url} 
\usepackage{mathabx,bm}
\usepackage{algorithm}
\usepackage[noend]{algpseudocode}
\usepackage{url}
\usepackage{enumitem}

\usepackage{geometry}
\theoremstyle{plain}
\newtheorem{thm}{Theorem}[section]
\newtheorem{prop}[thm]{Proposition}
\newtheorem{cor}[thm]{Corollary}

\theoremstyle{remark}
\newtheorem{rem}[thm]{Remark}

\theoremstyle{definition}
\newtheorem{ex}[thm]{Example}

\makeatletter

\numberwithin{equation}{section}
\numberwithin{figure}{section}
\numberwithin{table}{section}

\algnewcommand\algorithmicto{\textbf{to}}
\algdef{SE}[FOR]{ForTo}{EndFor}[2]{%
  \algorithmicfor\ #1\ \algorithmicto\ #2\ \algorithmicdo
}{%
  \algorithmicend\ \algorithmicfor
}

\title{\Large\bf Inverse Stochastic Control via Generalized Schr{\"o}dinger Problems}
\author{Yumiharu Nakano\\[1em]
        \small{Department of Mathematical and Computing Science, School of Computing} \\
        \small{Institute of Science Tokyo} \\
        \small{W8-28, 2-12-1, Ookayama, Meguro-ku, Tokyo 152-8550, Japan} \\
		\small{e-mail: nakano@comp.isct.ac.jp}
}

\date{\today}

\begin{document}

\maketitle

\begin{abstract}
We propose a variational formulation of an inverse problem in continuous-time stochastic control, 
aimed at identifying control costs consistent with a given distribution over trajectories. 
The formulation is based on minimizing the suboptimality gap of observed behavior.
We establish a connection between the inverse problem and a generalized dynamic Schr\"odinger problem, showing that their optimal values coincide. 
This result links inverse stochastic control with stochastic optimal transport, 
offering a new conceptual viewpoint on inverse inference in controlled diffusions.
\begin{flushleft}
\end{flushleft}
\end{abstract}





\section{Introduction}\label{sec:1}

Inverse stochastic control concerns the identification of latent control objectives that rationalize observed behavior
(see, e.g., Kalman \cite{kal:1964}, Bellman and Kalaba \cite{bel-kal:1963}, Ng and Russell \cite{ng-rus:2000}, 
Dvijotham and Todorov \cite{dvi-tod:2010}, and Nakano \cite{nak:2023}; 
see also Ab Azar et.al \cite{ab-etal:2020}, a recent review paper). 
In continuous-time settings, this problem naturally arises when one seeks to explain a given distribution over state--control trajectories by an underlying optimal control formulation.
Despite its relevance to data-driven control and inverse reinforcement learning, a general framework for such inverse problems remains limited.

The existing literature uses three main strategies.
The first one identifies costs by minimizing the residual of the optimality conditions of the forward problem. These conditions are the Karush-Kuhn-Tucker conditions \cite{kes-wan-boy:2011} or the Pontryagin minimum principle \cite{mol-for-per:2020}.
The second one comes from inverse reinforcement learning. It matches feature expectations \cite{abb-ng:2004} or maximizes an entropy-regularized likelihood \cite{zie-etal:2008} of the observed trajectories. Game-theoretic and margin-based formulations \cite{ng-rus:2000,sye-sch:2007} form a related line of work.
The third one comes from static data-driven inverse optimization. It minimizes a \emph{suboptimality loss}, which measures the gap between the cost of the observed decision and the optimal cost \cite{moh-etal:2018}. There it is used mainly as a convex and tractable surrogate, in a finite-dimensional deterministic setting.
None of these works addresses the continuous-time stochastic setting.

In this paper, we propose a variational formulation of inverse stochastic control in continuous time based on the \emph{suboptimality gap} of observed behavior. 
Given a prescribed distribution over trajectories and a class of candidate running and terminal costs, 
we quantify how far the observed behavior deviates from optimality under each candidate objective, and define the inverse problem as minimizing this gap. 
This formulation allows the analysis of observed distributions without assuming optimality of the underlying control policy.

Consider the following forward stochastic control problem: 
\begin{equation}
\label{eq:1.1}
 \inf_{u\in\mathcal{U}}J(u; f,g)
\end{equation}
where  
\[
 J(u; f,g)=\mathbb{E}\left[\int_0^Tf(t,X_t^u,u_t)dt + g(X_T^u)\right]. 
\]
The state process is described by the controlled diffusion
\begin{equation}
\label{eq:1.2}
 dX_t^u = b(t,X_t^u,u_t)dt + \sigma(t,X_t^u,u_t)dB_t.
\end{equation}
The control processes are assumed to take values in a closed set 
$U$ of $\mathbb{R}^{d_1}$. 
The precise definition of the class $\mathcal{U}$ of control processes is given in the Section \ref{sec:2} below.  

Suppose that a family of probability measures $\{\mu_t\}_{0\le t\le T}$ representing the observed state-control distribution is given.
Importantly, $\mu_t$ need not be realizable as the marginal law of any controlled process;
the formulation below is meaningful for any family of measures satisfying mild integrability conditions (see Section~\ref{sec:2}).
In the canonical case where $\mu_t$ is indeed the law of $(X_t^{u^*}, u_t^*)$ for some $u^*\in\mathcal{U}$,
the problem reduces to finding objective functions $f$ and $g$ such that $u^*$ is optimal for \eqref{eq:1.1}.

Let $\mathcal{D}$ be a given set of $(f,g)$ of functions such that $f$ is Borel measurable function on $[0,T]\times\mathbb{R}^d\times U$ and 
$g$ is Borel measurable function on $\mathbb{R}^d$. For $(f,g)\in\mathcal{D}$,  the functions $f$ and $g$ are interpreted as latent running cost and 
terminal cost functions in the forward optimal control problem \eqref{eq:1.1}, respectively.  
 To motivate the formulation, we argue informally as follows.
We introduce the functional
\begin{align*}
 V(f,g): &= \int_0^T\int_{\mathbb{R}^d\times U}f(t,x,u)\mu_t(dxdu)dt + \int_{\mathbb{R}^d}g(x)\tilde{\mu}_T(dx) - J^*(f,g)
\end{align*}
where $\tilde{\mu}_t$ denotes the spatial (state) marginal of $\mu_t$, that is, $\tilde{\mu}_t(dx):=\mu_t(dx\times U)$, and
\begin{equation}
\label{eq:1.3}
 J^*(f,g)=\inf_{u\in\mathcal{U}}J(u; f,g).
\end{equation}

If $(X_t^{u^*}, u_t^*)$ satisfies the controlled stochastic differential equation (SDE) above for some $u^*\in\mathcal{U}$, then
\[
 V(f,g)\ge 0
\]
for any $(f,g)\in\mathcal{D}$. 
Moreover, if $u^*$ is optimal to the forward control problem \eqref{eq:1.1} for some $(f^*,g^*)\in\mathcal{D}$, then 
\[
 V(f^*,g^*)=\inf_{(f,g)\in\mathcal{D}}V(f,g)=0. 
\]
This means that the minimization problem  
\begin{equation}
\label{eq:1.4}
 V^*:=\inf_{(f,g)\in\mathcal{D}}V(f,g) 
\end{equation}
can be a measure of the identification. 
From an intuitive viewpoint, $V(f,g)$ quantifies the sub-optimality gap 
of the observed behavior under the candidate cost $(f,g)$. 
If the gap vanishes, the observed trajectory is indeed optimal for $(f,g)$.
Hence, minimizing $V(f,g)$ provides a natural criterion for model identification,
analogous to minimizing a loss function in statistical learning.

The main contribution of the paper is a novel connection between this inverse control formulation and a \emph{generalized dynamic Schr\"odinger problem}. 
For the Schr\"odinger problem we refer to, e.g, Schr{\"o}dinger \cite{schr:1931} and \cite{schr:1932},  
Bernstein \cite{ber:1932}, Jamison \cite{jam:1974,jam:1975}, 
Dai Pra \cite{dai:1991}, Mikami \cite{mik:2004}, Nagasawa \cite{nag:1964}, Nelson \cite{nel:2020}, Zambrini \cite{zam:1986a}, L\'eonard \cite{leo:2014}, Chen, Georgiou and Pavon \cite{chen-etal:2016b,chen-etal:2021}, and the works cited therein.
We show that the optimal value of the inverse problem coincides with that of a stochastic control problem with relaxed marginal constraints. 
This result establishes a variational link between inverse stochastic control and stochastic optimal transport,
providing a new conceptual perspective on inverse inference in controlled diffusions.
We also establish existence of optimal solutions to the inverse problem under mild assumptions and discuss additional structural characterizations. 

The remainder of the paper is organized as follows. Section~II introduces the problem formulation and establishes existence results. 
Section~III presents the connection with generalized Schr\"odinger-type problems and derives the main variational equivalence.

\section{Existence}\label{sec:2}

First, we collect some notation used in this paper. 
Denote by $\mathcal{P}(\mathcal{X})$ the set of all Borel probability measures $\nu$ on a Polish space $\mathcal{X}$. 
We write $a^{\mathsf{T}}$ for the transpose of a vector or matrix $a$. 
For Euclidean spaces $\mathcal{X}$ and $\mathcal{Y}$, denote by $C(\mathcal{X}, \mathcal{Y})$ the space of all continuous functions $\varphi: \mathcal{X}\to\mathcal{Y}$. 
In particular, we denote $C(\mathcal{X})=C(\mathcal{X},\mathbb{R})$. 
We also denote by $C_b(\mathcal{X})$ the space of bounded continuous functions on $\mathcal{X}$.  
Given a Polish space $\mathcal{X}$, for any $\pi\in\mathcal{P}(\mathcal{X})$ and any Borel measurable function $\varphi:\mathcal{X}\to\mathbb{R}$ that is integrable
with respect to $\pi$, we put
\[
 \langle \varphi, \pi\rangle := \int_{\mathcal{X}}\varphi(x) \pi(dx).
\]
For a random variable $X$ and a probability measure $\nu$, we write $X\sim\nu$ to mean that $X$ is distributed according to $\nu$.
We denote by $\mathbb{W}^{d}:=C([0,T];\mathbb{R}^{d})$ the space of $\mathbb{R}^{d}$-valued continuous paths on $[0,T]$ (the path, or Wiener, space), equipped with the topology of uniform convergence.

Let $\{B_t\}_{0\le t\le T}$ be an $m$-dimensional standard 
Brownian motion on a complete probability space $(\Omega,\mathcal{F},\mathbb{P})$. 
Further, let $X_0$ be a random variable that is independent of $\{B_t\}_{0\le t\le T}$ such that 
\[
 \mathbb{E}|X_0|^2<\infty. 
\]
Denote by $\{\mathcal{F}_t\}_{0\le t\le T}$ the augmented filtration  
generated by $\{B_t\}_{0\le t\le T}$ and $X_0$. 
For simplicity, we assume $U\ni 0$.  
The class $\mathcal{U}$ of controls is then defined by the set of all 
$U$-valued $\{\mathcal{F}_t\}$-adapted processes $\{u_t\}_{0\le t\le T}$ satisfying 
\[
 \mathbb{E}\int_0^T|u_t|^2dt < \infty. 
\]
To discuss the uniqueness, we impose the following: 

\begin{enumerate}
\item[(A1)] There exists a constant $C_0>0$ and a modulus of continuity $\rho:[0,\infty)\to [0,\infty)$ such that 
for $\varphi(t,x,u) = b(t,x,u)$, $\sigma(t,x,u)$ we have   
\begin{align*}
 |\varphi(t,x,u)-\varphi(t,x^{\prime},u^{\prime})|
  &\le C_0|x-x^{\prime}| + \rho(|u-u^{\prime}|), \\ 
 |\varphi(t,0,0)|&\le C_0, 
\end{align*}
for $t\in [0,T]$, $x,x^{\prime}\in\mathbb{R}^d$, and $u,u^{\prime}\in U$. 
\end{enumerate}

Under (A1), for any given $u\in\mathcal{U}$, there exists a unique solution $\{X_t^u\}_{0\le t\le T}$ of the controlled stochastic differential equation
\eqref{eq:1.2} with initial condition $X_0$ such that
\begin{equation}
\label{eq:2.1}
 \mathbb{E}\left[\sup_{0\le t\le T}|X_t^u|^2\right]< \infty.
\end{equation}

Denote by $P_t^u \in \mathcal{P}(\mathbb{R}^d \times U)$ the probability distribution of $(X_t^u, u_t)$ under $\mathbb{P}$ for each $t\in [0,T]$.
Also denote by $\tilde{P}_t^u$ the law of $X_t^u$, i.e., $\tilde{P}_t^u(A)=P_t^u(A\times U)$, $A\in\mathcal{B}(\mathbb{R}^d)$.  
With these notation, 
\[
 V(f,g)= \int_0^T\langle f(t,\cdot,\cdot), \mu_t\rangle dt + \langle g, \tilde{\mu}_T\rangle - J^*(f,g). 
\]

\begin{prop}
Let $(A1)$ hold. 
Suppose that $\mu_t$ is the marginal law of $(X_t^{u^*}, u^*_t)$ for each $t\in [0,T]$, where $u^*=\{u^*_t\}_{0\le t\le T}\in\mathcal{U}$ is an optimal policy to the 
problem $\eqref{eq:1.1}$ for some $(f^*,g^*)\in\mathcal{D}$. Then, this $(f^*,g^*)$ is a minimizer for the inverse problem $\eqref{eq:1.4}$ and we have 
\[
 V(f^*,g^*)= \inf_{(f,g)\in\mathcal{D}}V(f,g) = 0. 
\]
\end{prop}

In the following analysis, we do not assume optimality of $\mu_t$. Even more, 
each $\mu_t$ may not be necessarily a marginal measure of $(X_t^u, u_t)$ for some $u\in\mathcal{U}$. However, we do assume the following conditions on 
$\mu_t$ and $\mathcal{D}$: 
\begin{enumerate}
\item[(A2)] The family $\{\mu_t\}_{0\le t\le T}$ satisfies 
\[
 \int_0^T\int_{\mathbb{R}^d\times U}(|x|^2 + |u|^2)\mu_t(dxdu)dt < \infty. 
\]
\end{enumerate}

\begin{enumerate}
\item[(A3)] There exists a constant $K>0$ such that for any $(f,g)\in\mathcal{D}$ we have 
\[
 |f(t,x,u)| + |g(x)|\le K(1+|x|^2 + |u|^2),  
\]
for $t\in [0,T]$, $x\in\mathbb{R}^d$, $u\in U$.
\end{enumerate}
Under (A1)--(A3), by \eqref{eq:2.1} it is straightforward to check that 
$V^*=\inf_{(f,g)\in\mathcal{D}}V(f,g)>-\infty$. 

Further, we assume the following: 
\begin{enumerate}
\item[(A4)] $U$ is compact. 
\end{enumerate}

For $(f,g)\in\mathcal{D}$, define $\phi:[0,T]\times\mathbb{R}^d\times U\to\mathbb{R}^2$ by 
\[
 \phi_{f,g}(t,x,u) = (f(t,x,u), g(x))^{\mathsf{T}}. 
\]
Then we assume the following: 
\begin{enumerate}
\item[(A5)]  The class $\{\phi_{f,g}: (f,g)\in\mathcal{D}\}$ is contained in $C([0,T]\times \mathbb{R}^d\times U, \mathbb{R}^2)$, is equi-uniformly continuous, and is closed under uniform convergence on compact sets.
\end{enumerate}
Conditions (A4) and (A5) are the compactness assumptions usually imposed in existence arguments of this kind. They let us extract a convergent subsequence of a minimizing sequence in the proof of Proposition~\ref{prop:2.4}.

\begin{ex}
Fix $N\in\mathbb{N}$, and assume that each $h=f,g$ is given by
\[
 h(\xi) = \sum_{j=1}^N\gamma_j e^{-\alpha_j |\xi- \xi_j|^2}, \quad \xi\in [0,T]\times\mathbb{R}^d\times U,
\]
where the parameters $(\gamma_j,\alpha_j,\xi_j)_{j=1}^N$ range over a fixed compact set with $\alpha_j\ge\alpha_0>0$. Then the class $\{\phi_{f,g}:(f,g)\in\mathcal{D}\}$ is the image of a compact set under a continuous map, hence compact, and it satisfies (A3) and (A5).
\end{ex}

Now we show the existence of a solution to the proposed inverse problem.
\begin{prop}
\label{prop:2.4}
Suppose that $(A1)$--$(A5)$ hold. Then there exists a minimizer of the inverse problem $\eqref{eq:1.4}$.  
\end{prop} 
\begin{proof}
Let $\{(f_n,g_n)\}_{n=1}^{\infty}\subset\mathcal{D}$ be a minimizing sequence, i.e., 
$\lim_{n\to\infty}V(f_n,g_n)= V^*$. 
By (A3) and (A5), the class $\{\phi_{f,g}: (f,g)\in\mathcal{D}\}$ is compact in $C([0,T]\times \mathbb{R}^d\times U, \mathbb{R}^2)$. 
It follows from the Ascoli--Arzel{\`a} theorem that there exist $(f^*, g^*)\in\mathcal{D}$ and $\{n_k\}$ such that 
$\sup_{(t,x,u)\in [0,T]\times B_R\times U}|f_{n_k}(t,x,u) - f^*(t,x,u)| + \sup_{x\in B_R}|g_{n_k}(x) - g^*(x)| \to 0$, $k\to\infty$,  
for any $R>0$, where $B_R$ stands for the closed ball in $\mathbb{R}^d$ centered at $0$ with radius $R$. 
Then, the dominated convergence theorem yields 
$\int_0^T\langle f_{n_k}(t,\cdot,\cdot), \mu_t\rangle \, dt \to \int_0^T\langle f^*(t,\cdot,\cdot), \mu_t\rangle \, dt$ and 
$\langle g_{n_k}, \tilde{\mu}_T\rangle\to \langle g^*, \tilde{\mu}_T\rangle$ 
as $k\to\infty$. 
Further, applying Theorem 12 in Section 1 of Chapter 3 in Krylov \cite{kry:1980}, we get 
\begin{align*}
 \varepsilon_k&:=\sup_{u\in\mathcal{U}}\mathbb{E}\int_0^T|f_{n_k}(t,X_t^u,u_t)-f^*(t,X_t^u,u_t)| \, dt + \mathbb{E}|g_{n_k}(X_T^u)- g^*(X_T^{u})|\to 0, 
\end{align*}
as $k\to\infty$. Therefore, 
$\lim_{k\to\infty}V(f_{n_k}, g_{n_k})= V(f^*,g^*)$. Thus $(f^*,g^*)\in\mathcal{D}$ is optimal. 
\end{proof}

\begin{ex}
Here we take $T=1$.
Consider the one-dimensional system
\[
 dX^u_t = u_t\,dt + \tfrac{1}{10}\,dB_t, \quad X_0 \sim N(0,1),
\]
with the parametric cost class $\mathcal{D}=\{(f_\theta, 0):\theta>0\}$, where $f_\theta(t,x,u)=10|x|^2+\theta|u|^2$.
For each $\theta>0$, the value function takes the form $v(t,x;\theta)=r_t(\theta)x^2+s_t(\theta)$,
where $r_t(\theta)$ satisfies the Riccati equation $\dot{r}_t = r_t^2/\theta - 10$, $r_1=0$,
with explicit solution
\[
 r_t(\theta) = \sqrt{10\theta}\tanh\bigl(\sqrt{10/\theta}\,(1-t)\bigr).
\]
Since $X_0\sim N(0,1)$, the optimal value is
\[
 J^*(f_\theta,0) = \sqrt{10\theta}\tanh\bigl(\sqrt{10/\theta}\bigr)
  - \tfrac{\theta}{100}\ln\cosh\bigl(\sqrt{10/\theta}\bigr).
\]
Suppose the observed data are generated by the optimal control $u_t^* = -(r_t(\theta^*)/\theta^*)X_t^{u^*}$
under the true parameter $\theta^*=1$.
By optimality, $V(f_{\theta^*},0)=0$,
and the strict convexity of the LQ problem implies $V(f_\theta,0)>0$ for all $\theta\ne\theta^*$.
Hence $V^*=0$ and $\theta^*=1$ is the unique minimizer of $\theta\mapsto V(f_\theta,0)$ over $\theta>0$,
confirming that the inverse problem \eqref{eq:1.4} correctly recovers the true parameter.
In practice the parameter is identified by minimizing the scalar map $\theta\mapsto V(f_\theta,0)$ over $\theta>0$. Here this map is available in closed form from the Riccati solution above, so the identification is a one-dimensional minimization.
\end{ex}

We remark that some uniqueness results are already obtained in \cite{nak:2023}. A regularization of our inverse problems is a subject of a future study. 

\section{Connection with stochastic optimal transport}\label{sec:3}

In this section, we give a connection between \eqref{eq:1.4} and stochastic optimal transport such as the Schr{\"o}dinger problem. 

We shall impose the following condition on $\mathcal{D}$: 
\begin{enumerate}
\item[(B1)]  The class $\mathcal{D}$ is represented as
\[
\mathcal{D}=(f_0,g_0) + \bigcup_{\lambda>0}\lambda \tilde{\mathcal{D}}
\]
where $f_0\in C([0,T]\times\mathbb{R}^d\times U)$ with $\inf_{(t,x,u)}f_0(t,x,u)>-\infty$,
$g_0\in C_b(\mathbb{R}^d)$, and
the class $\{\phi_{f,g}: (f,g)\in\tilde{\mathcal{D}}\}$ is non-empty, convex, symmetric at origin, and compact in $C_b([0,T]\times \mathbb{R}^d\times U, \mathbb{R}^2)$.
\end{enumerate}

\begin{ex}
A class of neural networks satisfies (B1).
More precisely, let $f_\theta(t,x,u)$ and $g_\theta(x)$ be neural networks with a fixed architecture
and bounded activation functions, so that $\phi_{f_\theta,g_\theta}\in C_b([0,T]\times\mathbb{R}^d\times U,\mathbb{R}^2)$,
parameterized by weights $\theta\in\Theta$, where $\Theta\subset\mathbb{R}^m$ is a compact, convex, and symmetric set.
If the map $\theta\mapsto\phi_{f_\theta,g_\theta}$ is linear (as when the hidden layer is fixed and $\theta$ collects the output weights of a single hidden layer with linear output and no output bias),
then the class $\{\phi_{f_\theta,g_\theta}:\theta\in\Theta\}$ is a compact, convex, and symmetric
subset of $C_b([0,T]\times\mathbb{R}^d\times U,\mathbb{R}^2)$, so (B1) is satisfied with $\tilde{\mathcal{D}}=\{(f_\theta,g_\theta):\theta\in\Theta\}$.
\end{ex}

For any $\mathcal{P}(\mathbb{R}^d\times U)$-valued processes $\pi=\{\pi_t\}_{0\le t\le T}$ and $\nu=\{\nu_t\}_{0\le t\le T}$, put 
\begin{align*}
 \rho(\pi,\nu)&= \sup_{(f, g)\in\tilde{\mathcal{D}}}\bigg\{\int_0^T\langle f(t,\cdot,\cdot), \pi_t-\nu_t\rangle\, dt  + \langle g, \tilde{\pi}_T-\tilde{\nu}_T\rangle \bigg\}. 
\end{align*}
Here, $\tilde{\pi}_t(dx)=\pi_t(dx\times U)$ and $\tilde{\nu}_t$ is similarly defined.
Note that under (B1), $\rho$ is a pseudometric on the set of $\mathcal{P}(\mathbb{R}^d\times U)$-valued processes.
Here $\rho$ is the support function of the symmetric, convex, and compact set $\tilde{\mathcal{D}}$. The symmetry gives $\rho\ge 0$ and $\rho(\pi,\nu)=\rho(\nu,\pi)$, the convexity gives the triangle inequality, and the compactness gives finiteness.
Consider the family of approximate bridge problems 
\[
 \inf_{u\in\mathcal{U}_{\rho, \varepsilon}}J(u; f_0,g_0), \quad \varepsilon\ge 0, 
\]
where $\mathcal{U}_{\rho,\varepsilon} = \{u\in\mathcal{U}: \rho(P^u, \mu)\le\varepsilon\}$. 
Unlike the classical Schr\"odinger bridge, the present formulation allows relaxed marginal constraints measured by the pseudometric $\rho$, 
which is induced by the class of admissible latent objectives.

Then we have the following: 
\begin{thm}
\label{thm:3.2}
Suppose that $(A1)$ and $(B1)$ hold. 
Suppose moreover that 
\begin{equation}
\label{eq:3.1}
 \inf_{u\in\mathcal{U}_{\rho, 0}}J(u; f_0,g_0) < \infty
\end{equation}
and that for any $(f,g)\in\mathcal{D}$ there exists a minimizer $u_{f,g}\in\mathcal{U}$ for the forward control problem $\eqref{eq:1.1}$. 
Then, $V^*$ is finite and 
\begin{align*}
 V^*&= \int_0^T\langle f_0(t,\cdot,\cdot), \mu_t\rangle\, dt 
         + \langle g_0, \tilde{\mu}_T\rangle - \lim_{\varepsilon\searrow 0}\inf_{u\in\mathcal{U}_{\rho, \varepsilon}}J(u; f_0,g_0). 
\end{align*}
\end{thm}
\begin{proof}[Proof of Theorem \rm{\ref{thm:3.2}}]
Step (i). Take an arbitrary element $(f_1,g_1)\in\tilde{\mathcal{D}}$ and $\lambda>0$.
Since $f_0$ is bounded below and $f_1$, $g_0$, $g_1$ are all bounded, we have
$J^*(f_0+\lambda f_1, g_0+\lambda g_1) >-\infty$ and so
$V^*< \infty$.

Let $u_0\in\mathcal{U}$ be such that $\rho(P^{u_0},\mu)=0$. Then, since $\tilde{\mathcal{D}}$ is symmetric at origin, for any $(f,g)\in\tilde{\mathcal{D}}$, 
\[
 \int_0^T\langle f(t,\cdot,\cdot), \mu_t - P_t^{u_0}\rangle\, dt + \langle g, \tilde{\mu}_T - \tilde{P}_T^{u_0}\rangle=0. 
\]
Thus, for any $(f,g)\in\tilde{\mathcal{D}}$, 
\begin{align*}
 &\int_0^T\langle f_0+\lambda f, \mu_t\rangle\, dt + \langle g_0+\lambda g, \tilde{\mu}_T\rangle - J^*(f_0+\lambda f, g_0+\lambda g) \\
 &\ge \int_0^T\langle f_0, \mu_t - P_t^{u_0}\rangle\, dt + \langle g_0, \tilde{\mu}_T- \tilde{P}_T^{u_0}\rangle. 
\end{align*}
Taking the infimum over $(f,g)\in\tilde{\mathcal{D}}$, we have 
\[
 V^*\ge \int_0^T\langle f_0, \mu_t - P_t^{u_0}\rangle\, dt + \langle g_0, \tilde{\mu}_T- \tilde{P}_T^{u_0}\rangle > -\infty. 
\]
Hence $V^*$ is finite. 

Step (ii). 
By (B1) and the bilinearity of the integral $\langle\cdot, \cdot\rangle$, clearly, the function 
\[
 \Psi[u, (f,g)] = \int_0^T\langle f, P_t^{u}-\mu_t\rangle\, dt + \langle g, \tilde{P}_T^{u}-\tilde{\mu}_T\rangle, 
\]
is concavelike in $\tilde{\mathcal{D}}$, i.e., 
for any $(f_1,g_1), (f_2,g_2)\in\tilde{\mathcal{D}}$, $\alpha\in [0,1]$, there exists $(f,g)\in\tilde{\mathcal{D}}$ such that 
$\alpha \Psi[u, (f_1,g_1)] + (1-\alpha)\Psi[u, (f_2,g_2)]\le \Psi[u, (f,g)]$ for any $u\in\mathcal{U}$. 
Moreover, the function $\Psi[u, (f,g)]$ is convexlike in $\mathcal{U}$, i.e., for any $u,\tilde{u}\in\mathcal{U}$, $\alpha\in [0,1]$, there exists $u^{\prime}\in\mathcal{U}$ 
such that $\alpha \Psi[u, (f,g)] + (1-\alpha)\Psi[\tilde{u}, (f,g)]\ge \Psi[u^{\prime}, (f,g)]$ for any $(f,g)\in\tilde{\mathcal{D}}$. Indeed, we can always take an optimal
$u_{f,g}$ for $u^{\prime}$.
Note that the existence of a minimizer $u_{f,g}$ for each $(f,g)\in\mathcal{D}$, assumed in Theorem~\ref{thm:3.2}, is precisely what guarantees this convexlike condition required by Sion's minimax theorem.
Also, by (B1), the set $\{\phi_{f,g}: (f,g)\in\tilde{\mathcal{D}}\}$ is compact.
Furthermore, the map $(f,g)\mapsto \Psi[u,(f,g)]$ is continuous with respect to the topology on $\tilde{\mathcal{D}}$:
indeed, $\Psi[u,(f,g)]$ is linear in $(f,g)$, and since the set $\{\phi_{f,g}:(f,g)\in\tilde{\mathcal{D}}\}$ is compact
in $C_b([0,T]\times\mathbb{R}^d\times U,\mathbb{R}^2)$, it is uniformly bounded, say $\|\phi_{f,g}\|_\infty\le M$
for all $(f,g)\in\tilde{\mathcal{D}}$.
Hence, by the bounded convergence theorem (applicable under assumptions (A2) and \eqref{eq:2.1}),
$\Psi[u,(f,g)]$ is norm-continuous in $(f,g)$ on $\tilde{\mathcal{D}}$,
and in particular upper and lower semicontinuous as required by Sion's theorem.
Hence, applying Sion's minimax theorem \cite[Theorem 4.2]{sio:1958}, we get
\[
\sup_{(f,g)\in\tilde{\mathcal{D}}}\inf_{u\in\mathcal{U}}\Psi[u, (f,g)] = \inf_{u\in\mathcal{U}}\sup_{(f,g)\in\tilde{\mathcal{D}}}\Psi[u, (f,g)], 
\]
whence 
\begin{align*}
 -V^* &=  \sup_{\lambda>0}\sup_{(f,g)\in\tilde{\mathcal{D}}}\inf_{u\in\mathcal{U}}\left\{\Psi[u,(f_0,g_0)] + \lambda \Psi[u, (f,g)]\right\} 
  = \sup_{\lambda>0}\inf_{u\in\mathcal{U}}\hat{J}_{\lambda}(u)
\end{align*}
where
\begin{align*}
 \hat{J}_{\lambda}(u) &= \int_0^T\langle f_0(t,\cdot,\cdot), P_t^u-\mu_t\rangle\, dt + \langle g_0, \tilde{P}_T^u-\tilde{\mu}_T\rangle + \lambda \rho(P^u, \mu), \quad u\in\mathcal{U}. 
\end{align*}
For any $\varepsilon>0$ and $\lambda>0$, 
\[
 \inf_{u\in\mathcal{U}}\hat{J}_{\lambda}(u)\le \inf_{u\in\mathcal{U}_{\rho,\varepsilon}}\Psi[u, (f_0,g_0)] + \lambda \varepsilon. 
\]
Letting $\varepsilon\to 0$ and then taking the supremum with respect to $\lambda$ in the inequality just above, we get 
\[
 -V^* \le \lim_{\varepsilon\to 0}\inf_{u\in\mathcal{U}_{\rho,\varepsilon}}\Psi[u, (f_0,g_0)]. 
\]

Step (iii). 
Let $\varepsilon>0$ be fixed. 
Let $\{\lambda_n\}_{n=1}^{\infty}$ be a positive sequence such that $\lim_{n\to\infty}\lambda_n=\infty$ and
\[
 \lim_{n\to\infty}\inf_{u\in\mathcal{U}}\hat{J}_{\lambda_n}(u) = \sup_{\lambda>0}\inf_{u\in\mathcal{U}}\hat{J}_{\lambda}(u). 
\]
Let $\{\varepsilon_n\}_{n=1}^{\infty}$ be a positive sequence such that $\varepsilon_n\to 0$, $n\to\infty$.
Then take $u_n\in\mathcal{U}$ satisfying  
\[
 \hat{J}_{\lambda_n}(u_n)\le \inf_{u\in\mathcal{U}}\hat{J}_{\lambda_n}(u) + \varepsilon_n 
\] 
for any $n\in\mathbb{N}$. 
We will show that 
\begin{equation}
\label{eq:3.2}
 \lim_{n\to\infty}\lambda_{n}\rho(P^{u_n},\mu) = 0. 
\end{equation}
Assume contrary that 
\[
 \limsup_{n\to\infty}\lambda_{n}\rho(P^{u_n},\mu) = 5\delta
\]
holds for some $\delta>0$. Then there exists a subsequence $\{n_k\}$ such that 
$\lim_{k\to\infty}\lambda_{n_k}\rho(P^{u_{n_k}},\mu) = 5\delta$.
Our assumption in the theorem means the existence of $u^*\in\mathcal{U}$ such that $\rho(P^{u^*},\mu)=0$. 
Thus $\hat{J}_{\lambda_n}(u_n)\le \hat{J}_{\lambda_n}(u^*) + \varepsilon_n$ for any $n$. 
In particular, the sequence $\{\hat{J}_{\lambda_{n_k}}(u_{n_k})\}_{k=1}^{\infty}$ is bounded, whence 
we can take a further subsequence $\{n_{k_m}\}$ such that 
\begin{equation*}
 \lim_{m\to\infty}\hat{J}_{\bar{\lambda}_m}(\bar{u}_m)=\kappa:=\limsup_{k\to\infty} \hat{J}_{\lambda_{n_k}}(u_{n_k}) < \infty, 
\end{equation*}
where $\bar{\lambda}_m=\lambda_{n_{k_m}}$ and $\bar{u}_m=u_{n_{k_m}}$. 
Then put $\bar{\gamma}_m=\rho(P^{\bar{u}_m},\mu)$. 
Now choose $m_0$ and $m_1$ such that 
$\kappa< \hat{J}_{\bar{\lambda}_{m_0}}(\bar{u}_{m_0}) + \delta$, $\hat{J}_{\bar{\lambda}_{m_1}}(\bar{u}_{m_1}) < \kappa +\delta$, 
$\bar{\lambda}_{m_1}> 7\bar{\lambda}_{m_0}$, and that 
$3\delta + \bar{\varepsilon}_{m_0}< \bar{\lambda}_{m_1}\bar{\gamma}_{m_1}< 7\delta$, 
where $\bar{\varepsilon}_m=\varepsilon_{n_{k_m}}$. 
With these choices it follows that 
\begin{align*}
 \kappa &< \hat{J}_{\bar{\lambda}_{m_0}}(\bar{u}_{m_0}) + \delta 
 \le \inf_{u\in\mathcal{U}}\hat{J}_{\bar{\lambda}_{m_0}}(u) + \bar{\varepsilon}_{m_0} + \delta
 \le \hat{J}_{\bar{\lambda}_{m_0}}(\bar{u}_{m_1}) + \bar{\varepsilon}_{m_0} + \delta \\ 
   &= \int_0^T\langle f_0(t,\cdot,\cdot), P_t^{\bar{u}_{m_1}}-\mu_t\rangle\, dt 
      + \langle g_0, \tilde{P}_T^{\bar{u}_{m_1}}-\tilde{\mu}_T\rangle
 + \bar{\lambda}_{m_0}\bar{\gamma}_{m_1} + \bar{\varepsilon}_{m_0} + \delta. 
\end{align*}
Observe 
\begin{align*}
 &\bar{\lambda}_{m_0}\bar{\gamma}_{m_1} + \bar{\varepsilon}_{m_0} + \delta< \frac{1}{7}\bar{\lambda}_{m_1}\bar{\gamma}_{m_1} + \bar{\varepsilon}_{m_0}  + \delta  
 < 2\delta + \bar{\varepsilon}_{m_0}< \bar{\lambda}_{m_1}\bar{\gamma}_{m_1} - \delta, 
\end{align*}
whence $\kappa < \hat{J}_{\bar{\lambda}_{m_1}}(\bar{u}_{m_1}) -\delta$, 
which is impossible.  

In particular, $\rho(P^{u_n},\mu)\le \varepsilon$ for a sufficiently large $n$. 
This yields 
\begin{align*}
&\inf_{u\in\mathcal{U}}\hat{J}_{\lambda_n}(u)\ge \hat{J}_{\lambda_n}(u_n) - \varepsilon_n 
\ge \inf_{u\in\mathcal{U}_{\rho,\varepsilon}}\Psi[u, (f_0,g_0)] + \lambda_n\rho(P^{u_n},\mu) - \varepsilon_n. 
\end{align*}
Letting $n\to\infty$ and then $\varepsilon\to 0$, we get 
\[
 -V^*\ge \lim_{\varepsilon\searrow 0}\inf_{u\in\mathcal{U}_{\rho,\varepsilon}}\Psi[u, (f_0,g_0)]. 
\]
Therefore the theorem follows.
\end{proof}

\begin{rem}
The decomposition in (B1) separates a fixed baseline $(f_0,g_0)$, which encodes prior knowledge about the latent cost, from the perturbation directions $\tilde{\mathcal{D}}$, and this is what enables the pseudometric $\rho$ and the change of variables used in the proof above. The decomposition is not unique, but $V^*$ depends only on the set $\mathcal{D}$, so the right-hand side of Theorem~\ref{thm:3.2} is invariant under admissible re-parameterizations.
The assumption that a minimizer $u_{f,g}\in\mathcal{U}$ exists for every $(f,g)\in\mathcal{D}$ is used to verify the convexlike condition on $\Psi[u,(f,g)]$ required by Sion's minimax theorem (Step~(ii) of the proof above).
Relaxing this to the existence of $\varepsilon$-optimal controls, or establishing it via measurable selection arguments under appropriate continuity conditions on $(f,g)\mapsto J^*(f,g)$, would broaden the scope of Theorem~\ref{thm:3.2} and is a subject of future work.
The theorem moreover yields a minimizing $(f,g)$, and not only the value $V^*$. Existence is given by Proposition~\ref{prop:2.4}, and the saddle point from Sion's theorem identifies the minimizing direction in $\tilde{\mathcal{D}}$ as the maximizer of the dual problem, realized by the optimal bridge control $u^*$ with $\rho(P^{u^*},\mu)=0$. In the Schr\"odinger case this is the optimal potential, with the optimal control in closed form (Corollary~\ref{cor:3.3}).
\end{rem}

\subsection*{Schr{\"o}dinger problem}

Here, we shall discuss a connection between Theorem \ref{thm:3.2} and the Schr{\"o}dinger bridge problems.
Consider the case where the forward control problem \eqref{eq:1.1} is described as
\[
 J(u) = \mathbb{E}\int_0^T|u_t|^2dt,
\]
with $U=\mathbb{R}^d$, where the controlled process $X_t^u$ is governed by
\[
 dX_t^u = u_tdt + dB_t, \quad X_0\sim \tilde{\mu}_0.
\]
The class $\mathcal{D}$ of latent objectives is given by
\[
 \mathcal{D}=\{(|u|^2, g): g\in C_b(\mathbb{R}^d)\}.
\]
Set $f_0(t,x,u)=|u|^2$, $g_0=0$, and let
\[
 \tilde{\mathcal{D}}=\{(0, g): g\in\mathcal{C}\},
\]
where $\mathcal{C}=\left\{g\in C_b(\mathbb{R}^d): \|g\|_{BL}\le 1\right\}$
with $\|g\|_{BL}=\sup_{x\in\mathbb{R}^d}|g(x)| + \sup_{x\neq y}|g(x)- g(y)|/|x-y|$.
We apply Theorem~\ref{thm:3.2} with $\mathcal{D}$ replaced by
$(f_0,g_0)+\bigcup_{\lambda>0}\lambda\tilde{\mathcal{D}}\subset\mathcal{D}$.
Since $f_0=|u|^2\ge 0$, the condition (B1) is satisfied:
$f_0$ is bounded below (hence the proof of Theorem~\ref{thm:3.2} applies),
$g_0=0\in C_b(\mathbb{R}^d)$, and the class $\{\phi_{f,g}:(f,g)\in\tilde{\mathcal{D}}\}$
is non-empty, convex, symmetric at origin, and compact
in the topology of uniform convergence on compact subsets of $\mathbb{R}^d$.
By the Kantorovich--Rubinstein theorem (see, e.g., Dudley \cite{dud:2002}),
\[
 \rho(\pi, \nu)= \sup_{g\in\mathcal{C}}\,\langle g, \tilde{\pi}_T-\tilde{\nu}_T\rangle
\]
is the bounded-Lipschitz metric on $\mathcal{P}(\mathbb{R}^d)$.
In particular, $\rho(\pi, \nu)=0$ if and only if $\tilde{\pi}_T=\tilde{\nu}_T$.
Consequently, our optimal transport problem becomes
\begin{equation}
\label{eq:3.3}
 \inf_{u\in\mathcal{U}_{\rho,0}}J(u) = \inf\left\{\mathbb{E}\int_0^T|u_t|^2dt: X_T^u\sim \tilde{\mu}_T, \; u\in\mathcal{U}\right\},
\end{equation}
which is the so-called Schr{\"o}dinger's bridge problem.

Now, using Theorem \ref{thm:3.2} we shall derive the duality formula for the Schr{\"o}dinger's problem, which is 
obtained in Mikami and Thieullen \cite{mik-thi:2006}. See Mikami \cite{mik:2021a}, \cite{mik:2021b} for further generalizations. 
It should be emphasized that our approach is completely different from ones used in \cite{mik-thi:2006}, \cite{mik:2021a}, and \cite{mik:2021b}.  
Recall from Section \ref{sec:2} that we have assumed 
\[
 \int_{\mathbb{R}^d}|x|^2\tilde{\mu}_0(dx) < \infty. 
\]
\begin{cor}[\cite{mik-thi:2006}]
\label{cor:3.3}
Suppose that $\tilde{\mu}_T$ has a positive density $\phi_T$ such that
\[
 \int_{\mathbb{R}^d}(|y|^2 + \log\phi_T(y))\phi_T(y)dy < \infty.
\]
Then we have
\begin{align*}
 &\sup_{g\in C_b(\mathbb{R}^d)}\left\{\inf_{u\in\mathcal{U}}\mathbb{E}\left[\int_0^T|u_t|^2dt + g(X_T^u)\right] - \langle g, \tilde{\mu}_T\rangle  \right\} 
 =  \inf\left\{\mathbb{E}\int_0^T|u_t|^2dt:  X_T^u\sim \tilde{\mu}_T, \; u\in\mathcal{U}\right\}.
\end{align*}
\end{cor}
\begin{proof}
Let $p(t,x,s,y)$ be the transition density of $d$-dimensional Brownian motion.   
Then, there exists a $\sigma$-finite product measure $\nu_0(dx)\nu_T(dy)$ such that 
\[
 \pi(E):= \int_Ep(0,x,T,y)\nu_0(dx)\nu_T(dy), \quad E\in\mathcal{B}(\mathbb{R}^d\times\mathbb{R}^d),
\]
satisfies $\pi(dx\times\mathbb{R}^d)=\tilde{\mu}_0(dx)$ and $\pi(\mathbb{R}^d\times dy)=\tilde{\mu}_T(dy)$ and 
\[
 H(\pi\;|\; p(0,x,T,y)\tilde{\mu}_0(dx)dy)< \infty,
\]
where $H(\pi\,|\, R)$ denotes the relative entropy of $\pi$ with respect to $R$. 
See e.g., Theorem 2.1 in Nutz \cite{nut:2022}. 
Then, by Theorem 2 in Jamison \cite{jam:1975}, there exists a weak solution of 
\[
 dX_t^* = \nabla h(t,X_t^*)dt + dB_t, \quad X_0^*\sim \tilde{\mu}_0, 
\]
where 
\[
 h(t,x) = \int_{\mathbb{R}^d}p(t,x,T,y)\nu_T(dy), \quad (t,x)\in [0,T)\times\mathbb{R}^d.
\]
We assume that our probability space, filtration, and Brownian motion are identical to those used to construct this weak solution. 
Then, the control process $\{u_t^*\}_{0\le t\le T}$ defined by $u^*_t=\nabla h(t,X_t^*)$ is in $\mathcal{U}$ and satisfies 
\[
 \inf_{u\in\mathcal{U}_{\rho,0}}J(u) = J(u^*) = 2 H(\pi\;|\; p(0,x,T,y)\tilde{\mu}_0(dx)dy),
\]
whence $u^*$ is optimal to the Schr{\"o}dinger's problem \eqref{eq:3.3}. 
In particular, the finiteness condition \eqref{eq:3.1} is satisfied.

Let us show 
\begin{equation}
\label{eq:3.4}
 \lim_{\varepsilon\searrow 0}\inf_{u\in\mathcal{U}_{\rho,\varepsilon}}J(u) = \inf_{u\in\mathcal{U}_{\rho,0}}J(u). 
\end{equation}
To this end, let $P^W, P^*\in\mathcal{P}(\mathbb{W}^d)$ be the law of $X_0+B$ and $X^*$, respectively. Then it is well-known that 
\begin{equation}
\label{eq:3.5}
 \inf H(Q\,|\, P^W)= H(P^*\,|\,P^W), 
\end{equation}
where the infimum is taken over all $Q\in\mathcal{P}(\mathbb{W}^d)$ with $Q_0=\tilde{\mu}_0$ and $Q_T=\tilde{\mu}_T$. Here $Q_0$ and $Q_T$ denote
the marginals of $Q$ at time $0$ and $T$, respectively.
Let $\varepsilon_n$ be a positive sequence such that $\varepsilon_n\searrow 0$. Let $u^{(n)}\in\mathcal{U}_{\rho,\varepsilon_n}$ be such that
\[
 J(u^{(n)})\le \inf_{u\in\mathcal{U}_{\rho,\varepsilon_n}}J(u) + \varepsilon_n.
\]
Put $X^{(n)} = X^{u^{(n)}}$. Let $\delta>0$ be arbitrary.
For a fixed $N\ge 1$ take $\{t_i\}$ so that $0=t_0<t_1<\cdots<t_N = T$ and $t_{\ell}-t_{\ell -1}=T/N$. Then from Theorem 7.4 in Billingsley \cite{bil:1999} it follows that
\begin{align*}
 &\mathbb{P}\left(\max_{\substack{0\le s<t\le T\\ 0<t-s<T/N}}|X_t^{(n)}-X_s^{(n)}|>\delta \right) \\
 &\le \mathbb{P}\left(\max_{\substack{0\le s<t\le T\\ 0<t-s<T/N}}\left|\int_s^t u_r^{(n)}dr\right|> \delta/2\right) 
 + \sum_{\ell=1}^N\mathbb{P}\left(\sup_{t_{\ell -1}\le s\le t_{\ell}}|B_s - B_{t_{\ell -1}}|\ge \delta/6\right).
\end{align*}
The first term of the right-hand side is at most
\begin{align*}
 &\frac{4}{\delta^2 N}\sup_n\mathbb{E}\int_0^T\left[1+|u_r^{(n)}|^2\right]dr 
 \le\frac{4}{\delta^2 N}\left[T + \sup_n\mathbb{E}\int_0^T|u_r^{(n)}|^2dr\right].
\end{align*}
The second term is at most $C_0\delta^{-3}N^{-1/2}$ where we have used Doob's maximal inequality.
Thus, for any $\delta>0$,
\[
 \lim_{N\to \infty}\sup_{n\ge 1}\mathbb{P}\left(\max_{\substack{0\le s<t\le T\\ 0<t-s<T/N}}|X_t^{(n)}-X_s^{(n)}|>\delta\right)=0.
\]
Therefore, by Theorem 7.3 in Billingsley \cite{bil:1999}, the sequence of laws of $X^{(n)}$, $n\in\mathbb{N}$, is tight.
Denote by $Q^{(n)}$ the probability distribution of $X^{(n)}$ under $\mathbb{P}$.
So there exist a subsequence $\{n_k\}$ and $\hat{P}\in\mathcal{P}(\mathbb{W}^d)$ such that $Q^{(n_k)}$ weakly converges to $\hat{P}$.
Since $\rho$ is a metric on $\mathcal{P}(\mathbb{R}^d)$ we have
\[
 \rho(\hat{P},\tilde{\mu})\le \rho(\hat{P}, Q^{(n_k)})
  + \rho(Q^{(n_k)},\tilde{\mu})\to 0,
\]
as $k\to\infty$, whence $\hat{P}_{T}=\tilde{\mu}_{T}$.
Further, by the lower semi-continuity of $Q\mapsto H(Q|P)$,
\begin{align*}
 &H(P^*\,|\, P^W) \le H(\hat{P}\,|\, P^W) 
 =\liminf_{k\to\infty} H(Q^{(n_k)}\,|\,P^W) 
 = \frac{1}{2}\liminf_{k\to\infty}J(u^{(n_k)})\le H(P^*\,|\, P^W).
\end{align*}
This means that $\hat{P}$ is optimal to \eqref{eq:3.5}, whence by uniqueness, we obtain $\hat{P}=P^*$.
What we have shown now is that each subsequence $\{Q^{(n_k)}\}$ contains a further subsequence
$\{Q^{(n_{k_j})}\}$ converging weakly to $P^*$.
Applying Theorem 2.6 in \cite{bil:1999}, we deduce that $\{Q^{(n_k)}\}$ converges weakly to $P^*$ as $k\to\infty$.
Again by the lower semi-continuity of the relative entropy,
\begin{align*}
 &\liminf_{n\to\infty}J(u^{(n)})
 =2\liminf_{n\to\infty}H(Q^{(n)}\,|\,P^W) 
 \ge 2H(P^*\,|\, P^W)\ge \limsup_{n\to\infty} J(u^{(n)}).
\end{align*}
Hence, 
\begin{align*}
 &J(u^*)=2H(P^*\,|\, P^W) = \lim_{n\to\infty}J(u^{(n)})
 \le \lim_{n\to\infty}\inf_{u\in\mathcal{U}_{\rho,\varepsilon}}J(u). 
\end{align*}
So \eqref{eq:3.4} holds.
Applying Theorem~\ref{thm:3.2} to $(f_0,g_0)+\bigcup_{\lambda>0}\lambda\tilde{\mathcal{D}}$, together with \eqref{eq:3.4}, gives
\[
 \sup_{\|g\|_{BL}<\infty}\Big\{\inf_{u\in\mathcal{U}}\mathbb{E}\Big[\int_0^T|u_t|^2dt + g(X_T^u)\Big] - \langle g,\tilde{\mu}_T\rangle\Big\} = \inf_{u\in\mathcal{U}_{\rho,0}}J(u),
\]
because $\bigcup_{\lambda>0}\lambda\mathcal{C}$ is the set of functions with finite bounded-Lipschitz norm. This supremum is over a subset of $C_b(\mathbb{R}^d)$, so it is at most the left-hand side of the corollary. For the reverse bound, take any $g\in C_b(\mathbb{R}^d)$ and any $u\in\mathcal{U}$ with $X_T^u\sim\tilde{\mu}_T$. Then $\mathbb{E}[g(X_T^u)]=\langle g,\tilde{\mu}_T\rangle$, and so
\[
 \inf_{u'\in\mathcal{U}}\mathbb{E}\Big[\int_0^T|u'_t|^2dt + g(X_T^{u'})\Big] - \langle g,\tilde{\mu}_T\rangle \le \mathbb{E}\int_0^T|u_t|^2dt.
\]
Taking the infimum over such $u$ shows that the left-hand side of the corollary is at most $\inf_{u\in\mathcal{U}_{\rho,0}}J(u)$. The two bounds give the corollary.

\end{proof}

The development above recovers the classical Schr\"odinger bridge as a special case of Theorem~\ref{thm:3.2}. Here $\tilde{\mathcal{D}}$ consists of terminal costs, so the constraint reduces to the single terminal marginal $X_T^u\sim\tilde{\mu}_T$. The general theorem also covers running-cost directions, where the term $\int_0^T\langle f,\pi_t-\nu_t\rangle\,dt$ in $\rho$ keeps the intermediate distributions $(X_t^u,u_t)\sim\mu_t$ close, and agreement of all time marginals gives a bridge with a full time-marginal constraint. Non-quadratic baseline costs $(f_0,g_0)$ are also covered.


\section*{Conclusion}\label{sec:4}

We have studied the inverse stochastic control problem from a variational perspective.
Given an observed state-control distribution $\{\mu_t\}$, which need not be realizable as the marginal law of any controlled process,
we introduced the suboptimality gap functional $V(f,g)$ as a natural criterion for identifying latent cost functions.
Under mild regularity conditions, we established the existence of a minimizer for the resulting inverse problem \eqref{eq:1.4} (Proposition~\ref{prop:2.4}).

The main contribution of the paper is a variational equivalence between the inverse problem and a generalized dynamic Schr\"odinger problem (Theorem~\ref{thm:3.2}):
the optimal value $V^*$ coincides with the value of a stochastic control problem with relaxed marginal constraints, measured by the pseudometric $\rho$ induced by the class $\tilde{\mathcal{D}}$ of latent objectives.
As a consequence, the classical duality formula for the Schr\"odinger bridge problem is recovered as a special case (Corollary~\ref{cor:3.3}), via a different approach from those in \cite{mik-thi:2006,mik:2021a,mik:2021b}.

Several directions remain open.
A regularization of the inverse problem, which would replace the hard constraint $\rho(P^u,\mu)=0$ by a penalized formulation, is a natural extension.
Uniqueness of the minimizer under general conditions, and computational methods for solving the inverse problem in practice, are also subjects of future work.
A comprehensive numerical study and a validation on real data are beyond the scope of this short paper, and we leave them to future work.


\subsection*{Acknowledgements}

The author is grateful to the anonymous reviewers for their careful reading and for their constructive comments, which have improved the manuscript.
This study is supported by JSPS KAKENHI Grant Number JP24K06861.

\bibliographystyle{plain}
\bibliography{../../../mybib}

\newcommand{\noop}[1]{}
\begin{thebibliography}{10}

\bibitem{ab-etal:2020}
N.~Ab~Azar, A.~Shahmansoorian, and M.~Davoudi.
\newblock From inverse optimal control to inverse reinforcement learning: A
  historical review.
\newblock {\em Annu. Rev. Control}, 50:119--138, 2020.

\bibitem{abb-ng:2004}
P.~Abbeel and A.~Y. Ng.
\newblock Apprenticeship learning via inverse reinforcement learning.
\newblock In {\em Proceedings of the 21st International Conference on Machine
  Learning}. ACM, 2004.

\bibitem{bel-kal:1963}
R.~Bellman and R.~Kalaba.
\newblock An inverse problem in dynamic programming and automatic control.
\newblock {\em J.\ Math.\ Anal.\ Appl.}, 7:322--325, 1963.

\bibitem{ber:1932}
S.~Bernstein.
\newblock Sur les liaisons entre les grandeurs al{\'e}atoires.
\newblock In {\em Proc.~Int.~Cong.~of Math.}, volume~1, pages 288--309, 1932.

\bibitem{bil:1999}
P.~Billingsley.
\newblock {\em Convergence of Probability Measures}.
\newblock John Wiley \& Sons, New York, 2nd edition, 1999.

\bibitem{chen-etal:2016b}
Y.~Chen, T.~Georgiou, and M.~Pavon.
\newblock On the relation between optimal transport and {S}chr{\"o}dinger
  bridges: {A} stochastic control viewpoint.
\newblock {\em J. Optim. Theory Appl.}, 169:671--691, 2016.

\bibitem{chen-etal:2021}
Y.~Chen, T.~T. Georgiou, and M.~Pavon.
\newblock Stochastic control liaisons: {R}ichard {S}inkhorn meets {G}aspard
  {M}onge on a {S}chr{\"o}dinger bridge.
\newblock {\em SIAM Rev.}, 63:249--313, 2021.

\bibitem{dai:1991}
P.~Dai~Pra.
\newblock A stochastic control approach to reciprocal diffusion processes.
\newblock {\em Appl.~Math.~Optim.}, 23:313--329, 1991.

\bibitem{dud:2002}
R.~M. Dudley.
\newblock {\em Real analysis and probability}, volume~74 of {\em Cambridge
  Studies in Advanced Mathematics}.
\newblock Cambridge University Press, Cambridge, 2nd edition, 2002.

\bibitem{dvi-tod:2010}
K.~Dvijotham and E.~Todorov.
\newblock Inverse optimal control with linearly-solvable {MDP}s.
\newblock In {\em Proc.\ International Conference on Machine Learning}, 2010.

\bibitem{jam:1974}
B.~Jamison.
\newblock Reciprocal processes.
\newblock {\em Z.~Wahrscheinlichkeitstheorie verw.~Gebiete}, 30:65--86, 1974.

\bibitem{jam:1975}
B.~Jamison.
\newblock The {M}arkov processes of {S}chr{\"o}dinger.
\newblock {\em Z.~Wahrscheinlichkeitstheorie verw.~Gebiete}, 32:323--331, 1975.

\bibitem{kal:1964}
R.~E. Kalman.
\newblock When is a linear control system optimal?
\newblock {\em Trans.\ ASME Ser.\ D: J.\ Basic Eng.}, 86:51--60, 1964.

\bibitem{kes-wan-boy:2011}
A.~Keshavarz, Y.~Wang, and S.~Boyd.
\newblock Imputing a convex objective function.
\newblock In {\em Proc.\ IEEE International Symposium on Intelligent Control},
  pages 613--619, 2011.

\bibitem{kry:1980}
N.~V. Krylov.
\newblock {\em Controlled diffusion processes}.
\newblock Springer-Verlag, New York, 1980.

\bibitem{leo:2014}
C.~L{\'e}onard.
\newblock A survey of the {S}chr{\"o}dinger problem and some of its connections
  with optimal transport.
\newblock {\em Discrete Contin.\ Dyn.\ Syst.}, 34(4):1533--1574, 2014.

\bibitem{mik:2004}
T.~Mikami.
\newblock Monge's problem with a quadratic cost by the zero-noise limit of
  $h$-path processes.
\newblock {\em Probab.~Theory Related Fields}, 129:245--260, 2004.

\bibitem{mik:2021a}
T.~Mikami.
\newblock Stochastic optimal transport revisited.
\newblock {\em SN Partial Differ. Equ. Appl.}, 2(1):5, 2021.

\bibitem{mik:2021b}
T.~Mikami.
\newblock {\em Stochastic optimal transportation: stochastic control with fixed
  marginals}.
\newblock Springer, Singapore, 2021.

\bibitem{mik-thi:2006}
T.~Mikami and M.~Thieullen.
\newblock Duality theorem for the stochastic optimal control problem.
\newblock {\em Stochastic Process.~Appl.}, 116(12):1815--1835, 2006.

\bibitem{moh-etal:2018}
P.~Mohajerin~Esfahani, S.~Shafieezadeh-Abadeh, G.~A. Hanasusanto, and D.~Kuhn.
\newblock Data-driven inverse optimization with imperfect information.
\newblock {\em Mathematical Programming}, 167:191--234, 2018.

\bibitem{mol-for-per:2020}
T.~L. Molloy, J.~J. Ford, and T.~Perez.
\newblock Online inverse optimal control for control-constrained discrete-time
  systems on finite and infinite horizons.
\newblock {\em Automatica}, 120:109109, 2020.

\bibitem{nag:1964}
M.~Nagasawa.
\newblock Time reversions of {M}arkov processes.
\newblock {\em Nagoya Math.~J.}, 24:177--204, 1964.

\bibitem{nak:2023}
Y.~Nakano.
\newblock Inverse stochastic optimal controls.
\newblock {\em Automatica}, 149:110831, 2023.

\bibitem{nel:2020}
E.~Nelson.
\newblock {\em Dynamical theories of Brownian motion}, volume 106.
\newblock Princeton university press, 2020.

\bibitem{ng-rus:2000}
A.~Y. Ng and S.~J. Russell.
\newblock Algorithms for inverse reinforcement learning.
\newblock In {\em Proc.\ International Conference on Machine Learning}, 2000.

\bibitem{nut:2022}
M.~Nutz.
\newblock Introduction to entropic optimal transport.
\newblock {\em Lecture notes, Columbia University}, 2022.

\bibitem{schr:1931}
E.~Schr{\"o}dinger.
\newblock {\"U}ber die umkehrung der naturgesetze.
\newblock {\em Sitzungsberichte Preuss. Akad. Wiss. Berlin. Phys. Math.},
  144:144--153, 1931.

\bibitem{schr:1932}
E.~Schr{\"o}dinger.
\newblock Sur la th{\'e}orie relativiste de l'{\'e}lectron et
  l'interpr{\'e}tation de la m{\'e}canique quantique.
\newblock {\em (French) Ann. Inst. H. Poincar{\'e}}, 2:269--310, 1932.

\bibitem{sio:1958}
M.~Sion.
\newblock On general minimax theorems.
\newblock {\em Pacific J. Math.}, 8:171--176, 1958.

\bibitem{sye-sch:2007}
U.~Syed and R.~E. Schapire.
\newblock A game-theoretic approach to apprenticeship learning.
\newblock In {\em Advances in Neural Information Processing Systems 20}, 2007.

\bibitem{zam:1986a}
J.~C. Zambrini.
\newblock Variational processes.
\newblock In S.~Albeverio, G.~Casati, and D.~Merlini, editors, {\em Stochastic
  Processes in Classical and Quantum Systems}, pages 517--529, Berlin,
  Heidelberg, 1986. Springer.

\bibitem{zie-etal:2008}
B.~D. Ziebart, A.~Maas, J.~A. Bagnell, and A.~D. Dey.
\newblock Maximum entropy inverse reinforcement learning.
\newblock In {\em Proc.\ AAAI Conference on Artificial Intelligence}, volume~8,
  pages 1433--1438, 2008.

\end{thebibliography}

\end{document}